\documentclass[a4paper,12pt,oneside]{amsart}

\usepackage{fullpage}

\usepackage{tikz}
\usepackage[pdftex]{hyperref}

\newcommand{\C}{\mathbb C}
\newcommand{\Z}{\mathbb Z}

\newtheorem{thm}{Theorem}[section]
\newtheorem*{thm*}{Theorem}

\newtheorem{prop}[thm]{Proposition}

\theoremstyle{definition}
\newtheorem{defn}[thm]{Definition}

\newtheorem{rem}[thm]{Remark}
\newtheorem{ex}[thm]{Example}

\newcommand{\oct}{\operatorname{Oct}}

\title{An easy (horizontal) walk through fake octagons}
\author{Davide Dobrilla}

\author{Stefano Francaviglia}
\address{Dipartimento di Matematica dell'Università di Bologna, Piazza di Porta S. Donato 5,
  40126 Bologna, Italy}
\email{davide1190@gmail.com}
\email{stefano.francaviglia@unibo.it}

\begin{document}
\begin{abstract}
A fake octagon is a genus two translation surface with only one singular point and the same
periods as the octagon. Existence of infinitely many fake octagons was established first by
McMullen~\cite{Mcm07} in 2007, and more generally follows from dynamical properties of the isoperiodic foliation.

The purpose of this note is to describe an infinite family
of fakes constructed by means of elementary methods. We describe an easy cut-and-paste surgery
and show that the $n^{th}$ iterate of that surgery is a fake octagon $\oct_n$. Moreover we
show that $\oct_n\neq \oct_m$ for $n\neq m$, and that any $\oct_n$ can be approximated
arbitrarily well by some other $\oct_m$. This note is intended to be elementary and fully
accessible to non-expert readers.         
\end{abstract}

\maketitle

\section{Introduction}

Bibliography on translation surfaces is immense, we cite here only the celebrated handbooks
of dynamical systems (see for instance~\cite{esk,for,hub,masur,masur2}), the nice
survey~\cite{wri}, as well as~\cite{yoc} and~\cite{zor}, and references therein. Also,
we refer to Section~\ref{s2} for precise definitions, staying colloquial in this introduction.

The translation surface obtained by gluing parallel sides of a regular octagon is commonly
known as ``the octagon''. A fake  octagon is a translation surface with one singular point and
the same periods as the octagon.

It is well known that periods are local coordinates for the moduli space of translations
surfaces of fixed genus and singular divisor. Periods come in two flavours: absolute and
relative: former ones are translation vectors associated to closed loops, the latter are those
associated to saddle connections (i.e. paths connecting singular points). So-called isoperiodic
deformations of a translation surface consist in changing
relative periods without touching absolute ones. Isoperiodic loci are leaves of the isoperiodic
foliation (also known as absolute period foliation or kernel foliation). Local coordinates
on isoperiodic leaves are given by positions of singular points with respect to a fixed
singular point, chosen as origin. As a consequence, translation surfaces of the minimal stratum (that is, with
a unique singular point) cannot be continuously and isoperiodically deformed in that
stratum (all periods are absolute).

A priori, it is not clear whether or not, given $X$ in the minimal stratum, there is
a translation surface, still in the minimal stratum, with same periods as $X$. If any, such surfaces are
called ``fake $X$''.  In fact, the question of finding fakes of famous translation surfaces, as for
instance the octagon,  was a nice coffee-break problem in dynamical system
conferences some years ago. Nowadays, this is literature.

\medskip

Fakes where introduced and studied 
by McMullen in~\cite{Mcm07,Mcm14} --- who gave a complete and detailed description
of isoperiodic leaves in genus two --- and dynamical properties of the isoperiodic foliation were
established in~\cite{CDF,ursula} in general (in particular ergodicity and classification of
leaf-closures).

From~\cite{Mcm07,Mcm14,CDF,ursula} it follows that if periods of $X$
are not discrete (e.g. the octagon), then $X$ has infinitely many fakes. More precisely, the
isoperiodic leaf through $X$ intersects the minimal stratum $\mathcal H_{2g-2}$ in a set whose
closure has positive dimension. In particular, any such $X$ can be approximated by 
fakes. Moreover, in~\cite{Mcm14} McMullen showed that, in genus two, fakes are arranged in
horizontal strips, and described all fake pentagons.

\medskip
The purpose of this note is to give easy proofs of such results for the particular case of the
octagon by using elementary methods; where ``easy'' means ``explicable in a conference
coffee-break''. The ``elementary methods'' we use are surgeries that are the topological
viewpoint of the so-called Schiffer variations. Given the octagon, we describe a surgery (that
we call ``left-surgery'') that produces a fake octagon and that can be iterated. We then
prove that all fakes produced by iterating left-surgeries are in fact different from each
other, exhibiting  therefore an explicit infinite family of fake octagons. Also, we will show
that any fake of the family can be arbitrarily approximated by iterations. We note
that all our fakes are along a ``horizontal'' line of the isoperiodic leaf of the octagon: the
Schiffer variations are always in the horizontal direction. In this way we describe all fakes
in a horizontal strip. 

Finally, we discuss ingredients
needed for possible generalisations. Our main result is summarised as follows:

\begin{thm*}[Theorem~\ref{thm1}, Remark~\ref{r44}]
  Fake octagons obtained by iterated left-surgeries on the octagon are different
  from each other, and any such fake can be arbitrarily approximated by iterates.
\end{thm*}

\medskip

\noindent{\bf Acknowledgements} This work originated from master's thesis~\cite{dob} of first named
author. Second named author would like to thank first named author for the genuine friendship born during the redaction of that thesis.

\section{Isoperiodic foliation and fakes}\label{s2}
Translation structures on closed, connected, oriented surfaces can be defined in many different ways, for instance:

\begin{itemize}
\item They can be viewed as Euclidean structures with cone-singularities of cone-angles
  multiple of $2\pi$, up to isometries that reads as translations in local
  charts. Equivalently, they are branched $\C$-structures whose holonomy consits of
  translations,  where ``branched'' means that the developing map is not just a local
  homeomorphism but can also be a local branched covering;
\item or as pairs $(X,\omega)$ where $X$ is a Riemann surface and $\omega$ a holomorphic
  $1$-form, up to biholomorphisms;
\item or quotients of poligons in $\C$ via gluings that identify pairs of parallel edges via
  translations, up to suitable ``tangram'' relations. 
\end{itemize}

Third construction clearly produces a Euclidean structure with cone-singularities,
which, by pulling back the structure of $(\C,dz)$ produces a complex structure together with a
$1$-form (whose zeroes correspond to cone-singularities). In fact,  it turns out that all viewpoints are equivalent (we refer to~\cite{wri}
for more details). Any singular point has an order: if viewed as a cone-point, then it has order $d$
if the total angle is $2\pi+2\pi d$; if viewed as a zero of $\omega$, then it has order $d$ if
locally $\omega=z^ddz$. 

As usual, we will refer to a surface endowed with a translation structure as a {\em
  translation surface}. Singular points are also referred to as {\em saddles}.

If a translation surface has genus $g$, then by Gauss-Bonnet (or by a characteristic count) the
sum of the orders of singular points is $2g-2$.

The moduli space of translation surface of genus $g$ --- that we denote simply by $\mathcal H$ if
there is no ambiguity on the genus ---  is naturally stratified by the singular
divisor: if $\kappa$ is a partition of $2g-2$ (more precisely a list of non increasing positive integers
summing up to $2g-2$) then the stratum $\mathcal H(\kappa)$ consists of all translation surfaces
whose singular points have orders as prescribed by $\kappa$. For example, in genus $g=2$ there are
only two strata: the principal, or generic, stratum $\mathcal H_{1,1}$ --- consisting of
translation surfaces with two simple singular points (with cone-angles $4\pi$ each) --- and the
minimal stratum $\mathcal H_2$ --- consisting  of translation surfaces having only one singular
point of cone-angle $6\pi$. It turns out that any stratum is a complex orbifold of dimension
$2g+s-1$ where $s=|\kappa|$ is the number of singular points.

Apart from obvious issues due to the orbifold structure, periods give coordinates on any stratum. More
precisely, if $S$ is a translation surface with singular locus $\Sigma=\{x_1,\dots,x_s\}$, then
we consider the relative homology $H_1(S,\Sigma;\Z)$. If
$\gamma_1,\dots,\gamma_{2g}$ is a basis of $H_1(S;\Z)$ and  $\eta_2,\dots, \eta_s$
are arcs connecting $x_1$ to $x_2,\dots, x_s$, then the family
$\gamma_1,\dots,\gamma_{2g},\eta_2,\dots,\eta_{s}$ is a basis of $H_1(S,\Sigma;\Z)$.
By using
the $(X,\omega)$ viewpoint of translation surface, the period map
$$(X,\omega)\mapsto (\int_{\gamma_1}\omega,\dots,\int_{\gamma_{2g}}\omega,\int_{\eta_2}\omega,\dots,\int_{\eta_s}\omega)$$

is a local chart $\mathcal H(\kappa)\to \C^{2g+s-1}$. These are the so called {\bf period
  coordinates}. In other words, we consider $[\omega]\in H^1(S,\Sigma;\C)$.  Periods
of curves $\gamma_i$'s are usually called {\bf absolute periods}, while those of $\eta_i$'s are
{\bf relative periods}.

There is a natural period map $Per:\mathcal H\to \C^{2g}=H^1(S;\C)$ that associates to any
translation surface its absolute periods
$$Per: (X,\omega)\mapsto (\int_{\gamma_1}\omega,\dots,\int_{\gamma_{2g}}\omega)$$

The so-called {\bf isoperiodic foliation} $\mathcal F$ (also known as {\em kernel foliation} or {\em
  absolute period foliation}) is the foliation locally defined by the fibers of $Per$. Namely,
two translation surfaces are in the same leaf of $\mathcal F$ if one can be continuously deformed
into the other without changing absolute periods. The isoperiodic foliation is globally
defined in $\mathcal H=\cup_\kappa\mathcal H(\kappa)$, and its leaves have dimension
$2g-3$. Isoperiodic foliation has been extensively studied, for instance  in~\cite{Mcm07,Mcm14,CDF,ursula,bain,lef,ygouf}. 

One of the problems in studying isoperiodic foliation, is to determine the foliation induced by
$\mathcal F$ on each stratum. For instance, in the minimal stratum $\mathcal H_{2g-2}$ there is
no room for deformations: locally, any leaf of $\mathcal F$ intersects transversely such stratum
in a single point. Given $X\in\mathcal H_{2g-2}$, a ``{\bf fake} $X$'' is a translation
surface, different from $X$, but
with same absolute periods as $X$ (as a polarized module) and only one singular point, that is to say, if $F_X$ is the
leaf of $\mathcal F$ through $X$, then  a ``fake $X$'' is a point in $F_X\cap \mathcal H_{2g-2}$.

\begin{ex}
The so-called {\em octagon} is 
the translation surface obtained by gluing parallel sides of a regular octagon
sitting in $\C$ with an edge in the segment $[0,1]$. It is a genus two surface with a single
singular point. A {\em fake octagon} is an intersection point
of the isoperiodic leaf of the octagon with the minimal stratum $\mathcal H_2$, i.e. any
translation surface with the same (absolute) periods as the octagon (the
same area) and only one singular point.
  
\end{ex}

\section{Traveling on isoperiodic leaves by moving singular points}
If $X$ has $s$ singular points, then there are $s-1$ degrees of freedom for perturbing $X$
without changing its absolute periods (we can change the relative periods of
$\eta_2,\dots,\eta_s$). It turns out that local parameters are exactly the positions of singular
points; more precisely, the relative positions of $x_2,\dots, x_s$ with respect to $x_1$. So we
can travel the isoperiodic leaf through $X$ by ``moving'' singular points. From an analytic
viewpoint such moves are known as Schiffer variations (see~\cite{schi,BPS}). We adopt here a more topological
cut-and-paste viewpoint. We briefly recall the basic construction, referring to~\cite{CDF,BPS}
for a more detailed discussion. 

Let $x$ be a singular point and let $\gamma$ be a segment, or more generally a path, starting
at $x$. If $x$ has degree $d$, then $\gamma$ has $d$ {\bf twins}, that is to say, paths
starting at $x$ with the same developed image as $\gamma$ (by simplicity we assume here that
none of such twins contains a saddle in its interior). Explicitly, if $\gamma$ is a
segment, its twins are segments forming angles $2\pi, 4\pi,\dots,d2\pi$ with $\gamma$. For any
twin of $\gamma$ we can perform a cut-and-paste surgery as follows: We cut along $\gamma$ and
the chosen twin, and then we glue in the unique other way coherent with orientations. This is
better described in Figure~\ref{fig:twins}.
\begin{figure}[h]
\footnotesize  \centering
  \begin{tikzpicture}[x=1ex,y=1ex]
    \draw[->] (0,0)--(-7,0);\draw(-7,0)--(-10,0);
    \foreach \a in{45,90,-45,-90,0,135,-135}{
    \begin{scope}[rotate={\a}]
            \draw[dotted, ->] (0,0)--(7,0);\draw[dotted](7,0)--(10,0); 
    \end{scope}}
  \draw (-3,0) arc (180:135:3);
  \draw[fill] (0,0) circle[radius=.3];
  \node at (-5,2) {$2\pi$};
  \node[below] at (-10,0) {$\gamma$};
  \node[below] at (-1,-10) {\parbox{15ex}{$\gamma$ and its twins}};

  \begin{scope}[shift={(27,0)}]
    \draw[->] (0,0)--(-7,0);\draw (-7,0)--(-10,0);
    \foreach \a in{90,-45,-90,0,135,-135}{
    \begin{scope}[rotate={\a}]
      \draw[dotted, ->] (0,0)--(7,0);\draw[dotted](7,0)--(10,0);
      
    \end{scope}}
    \begin{scope}[rotate={45}]
      \draw[->] (0,0)--(7,0);\draw(7,0)--(10,0);
    \end{scope}

  \draw (-3,0) arc (180:45:3);
  \draw[fill] (0,0) circle[radius=.3];
  \node at (-1.5,4) {$6\pi$};
  \node[below] at (-10,0) {$\gamma$};
  \node[below] at (-1,-10) {\parbox{22ex}{$\gamma$ and one chosen twin}};
      \end{scope}

      \begin{scope}[shift={(54,0)}]
        \draw[fill, gray!7] (0,0)--(-10,0)--(-2,5)--(7,7)--(0,0);
        \draw(0,0)--(-10,0)--(-2,5)--(7,7)--(0,0);
    \draw[->] (0,0)--(-7.5,0);
    \draw[->] (-2,5)--(-8,1.25);
    \draw[->] (-2,5)-- (4.75,6.5);
    
    \foreach \a in{-45,-90,0,-135}{
    \begin{scope}[rotate={\a}]
      \draw[dotted, ->] (0,0)--(7,0);\draw[dotted](7,0)--(10,0);
      
    \end{scope}}
    \foreach \a in{90,135}{
    \begin{scope}[shift={(-2,5)},rotate={\a}]
      \draw[dotted, ->] (0,0)--(7,0);\draw[dotted](7,0)--(10,0);
    \end{scope}}
    \begin{scope}[rotate={45}]
         \draw[->] (0,0)--(7.5,0);\draw (7.5,0)--(10,0);
    \end{scope}

    \draw[fill] (0,0) circle[radius=.3]; \draw[fill] (-2,5) circle[radius=.3];
      \draw (-10,0) circle[radius=.3]; \draw (7,7) circle[radius=.3];

  \node[below] at (0,-10) {\parbox{7ex}{cutting\dots}};
      \end{scope}

  \begin{scope}[shift={(81,0)}]
    \draw[->] (0,0) -- (-3,4.5);\draw(-3,4.5)--(-4,6);
    \draw[->] (-8,12) -- (-5,7.5);\draw(-5,7.5)--(-4,6);
    \foreach \a in{-45,-90,0,-135}{
    \begin{scope}[rotate={\a}]
      \draw[dotted, ->] (0,0)--(3,0);\draw[dotted](3,0)--(4.5,0);
      
    \end{scope}}
    \foreach \a in{90,135}{
    \begin{scope}[shift={(-8,12)},rotate={\a}]
      \draw[dotted, ->] (0,0)--(3,0);\draw[dotted](3,0)--(4,0);
    \end{scope}}

  \draw[fill] (0,0) circle[radius=.3]; \draw[fill] (-8,12) circle[radius=.3];
  \draw (-4,6) circle[radius=.3];
  \node[below] at (0,-10) {\parbox{15ex}{\dots and pasting}};
      \end{scope}

    \end{tikzpicture}
  \caption{Moving singular points via cut-and-paste surgeries}
  \label{fig:twins}
\end{figure}
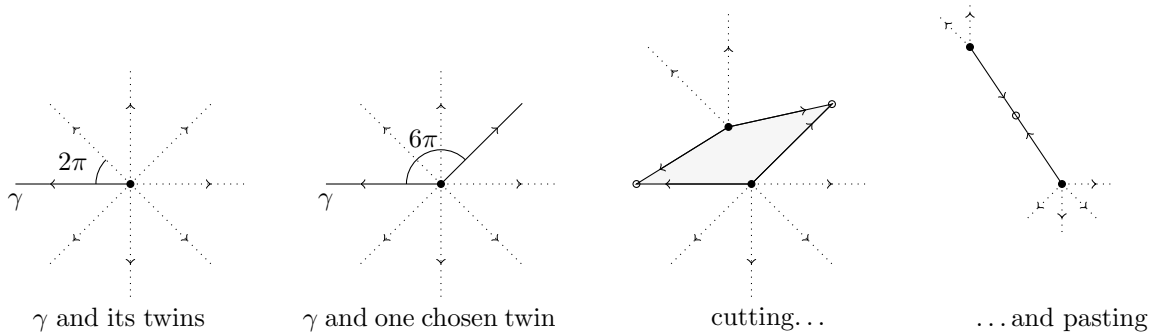

A first remark on that surgery, is that endpoints of $\gamma$ and the twin can be both
regular, both singular, or one regular and the other singular point. Given the angles at
endpoints, and the angle between $\gamma$ and its twin, we can easily recover angles after the
surgery (see Figure~\ref{fig:ang}):
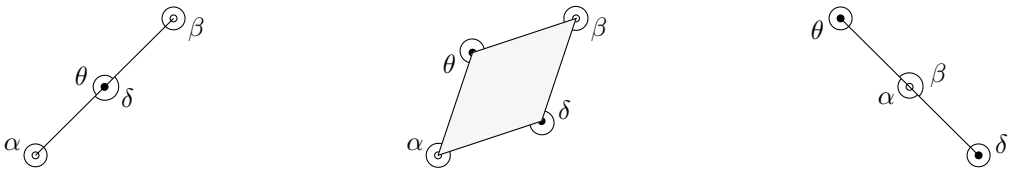
\begin{figure}[h]
  \footnotesize \centering
  \begin{tikzpicture}[x=1ex,y=1ex]
    \draw(-6,-6) circle[radius=.3]; \draw (6,6) circle[radius=.3];
    \draw[fill] (0,0) circle[radius=.3];
    \draw (-6,-6)--(6,6);
    \draw (6,6) circle [radius=1];     \draw (-6,-6) circle [radius=1];
    \draw (.72,.72) arc (45:225:1);
    \draw (-.82,-.82) arc (-135:45:1.2);

    \node at (-8,-5) {$\alpha$};
    \node at (8,5) {$\beta$};
    \node at (-2,1) {$\theta$};
    \node at (2,-1) {$\delta$};
    
    \begin{scope}[shift={(35,0)}]
      \draw(-6,-6) circle[radius=.3]; \draw (6,6) circle[radius=.3];
      \draw[fill] (3,-3) circle[radius=.3];\draw[fill] (-3,3) circle[radius=.3];
      \draw[fill, gray!7] (-6,-6)--(-3,3)--(6,6)--(3,-3)--(-6,-6);
      \draw (-6,-6)--(-3,3)--(6,6)--(3,-3)--(-6,-6);

      \draw (5.66,5) arc (-108:198:1.05);
      \draw (-5.66,-5) arc (-288:18:1.05);

      \draw (2,-3.33) arc (-170:74:1.05);
      \draw (-2,3.33) arc (10:254:1.05);

    \node at (-8,-5) {$\alpha$};
    \node at (8,5) {$\beta$};
    \node at (-5,2) {$\theta$};
    \node at (5,-2) {$\delta$};
    
    \end{scope}

    \begin{scope}[shift={(70,0)}]
      \draw[fill](-6,6) circle[radius=.3]; \draw[fill] (6,-6) circle[radius=.3];
      \draw (0,0) circle[radius=.3];
      \draw (-6,6)--(6,-6);
       \draw (-6,6) circle [radius=1];     \draw (6,-6) circle [radius=1];
    \draw (-.72,.72) arc (135:315:1);
    \draw (.82,-.82) arc (-45:135:1.2);

    \node at (8,-5) {$\delta$};
    \node at (-8,5) {$\theta$};
    \node at (-2,-1) {$\alpha$};
    \node at (2.5,1) {$\beta$};

  \end{scope}

  \end{tikzpicture}
  \caption{Angles before and after surgery}
  \label{fig:ang}
\end{figure}

In Figure~\ref{fig:ang}, before the surgery the full-dotted singular point has total angle
    $\theta+\delta$, and after it splits in two points. The two empty-dotted points paste
  together to form a point of total angle $\alpha+\beta$. All $\alpha,\beta,\theta,\delta$ are
  multiple of $2\pi$ (they are $2\pi$ precisely when the  corresponding point is regular).

Note that our surgeries take place locally, near a singular point. It follows that they
do not affect absolute periods (wile clearly they affect relative periods). It turns out that
these moves are the only way to isoperiodically deform a translation
surface. (See~\cite{CDF,BPS}). 

It maybe useful to remark at this point that such surgeries may or may not preserve strata. With
notations as in Figure~\ref{fig:ang}, if $\alpha,\beta,\theta,\delta$ are all $2\pi$, then what
we are doing is to move a singular point from the starting point of $\gamma$ to its endpoint
(in this case the stratum does not change).

If $\delta,\theta >2\pi$, and $\alpha,\beta=2\pi$, then we are 
splitting a singular point in two separate singular points and creating a singular point of
angle $4\pi$. (The sum of resulting degrees equals that of
initial ones). So in this case we are changing stratum.

Similarly, if for instance $\alpha=4\pi$, and $\theta,\delta,\beta=2\pi$, the surgery collapses together
two singular points, hence again changing stratum. There are more possibilities, and other kind of
surgeries are possible (for instance by cut and pasting along many twins simultaneously). We
refer the interested reader to~\cite{BPS,CDF} for further details.

The last needed remark, is that it may happen that $\gamma$ is a loop, starting and
ending at  the same point.  In this case twins of $\gamma$ may or may be not loops, and
conversely. Also, it can  even happen that  $\gamma$ is embedded, but the twin is not. In such
cases some topological disasters may happen (the surgery could for instance disconnect the
surface) and one has to check what happens carefully.

We will use surgeries where $\gamma$ is a closed {\bf saddle connection}, that is to say a straight
segment starting and ending at the same singular point, but we will always require that twins
of $\gamma$ are embedded segments. It is readily checked that in this case no disasters occur.
We refer to such a cut-and-paste as {\bf saddle connection surgery}. See Figure~\ref{fig:s}.
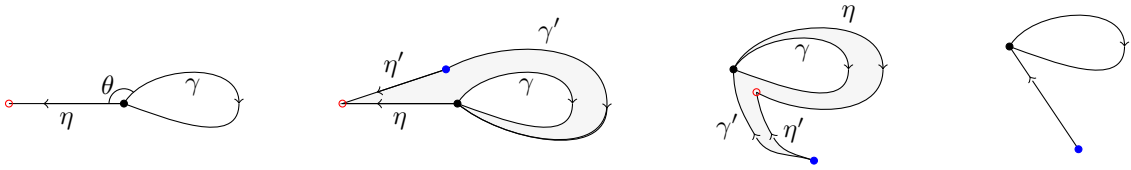
\begin{figure}[h]
  \footnotesize\centering
  \begin{tikzpicture}[x=1ex,y=1ex]
    \draw[->] (-10,0)--(0,0) to[in=90,out=60] (10,0); \draw (10,0)to[out=-90,in=-20] (0,0);
    \draw[->] (0,0)--(-7,0);
    \draw[fill] (0,0) circle [radius=.3];
    \draw[red] (-10,0) circle [radius=.3];
    \node[below] at (-5,0){$\eta$};
    \node[below] at (6,3){$\gamma$};
    \draw (-1.3,0) arc (180:50:1.3);
    \node[above left] at (0,0){$\theta$};
    
    \begin{scope}[shift={(29,0)}]
   \draw[fill, gray!7] (-10,0)--(0,0) to[in=90,out=60] (10,0) to[out=-90,in=-20] (0,0)
   to[in=-90,out=-40] (13,0) to [out=90,in=30] (-1,3)--(-10,0);

    \draw[->] (-10,0)--(0,0) to[in=90,out=60] (10,0); \draw (10,0)to[out=-90,in=-20] (0,0);
    \draw[->] (0,0)--(-7,0);

    \draw (0,0)to[out=-40,in=-90] (13,0) to[out=90,in=30] (-1,3)--(-10,0)--(0,0);
    \draw[-<] (0,0)to[out=-40,in=-90] (13,0);
    \draw[->] (-1,3)--(-7,1);
    \draw[fill] (0,0) circle [radius=.3];
    \draw[red] (-10,0) circle [radius=.3];
    \draw[fill, blue] (-1,3) circle [radius=.3];
    \node[below] at (-5,0){$\eta$};
    \node[below] at (6,3){$\gamma$};
    \node[above] at (-5.5,1.5){$\eta'$};
    \node[above] at (8,4.5){$\gamma'$};
    \end{scope}

    \begin{scope}[shift={(53,3)}]

      \draw[fill, gray!7](0,0) to[out=65,in=90] (13,0)
      to[out=-90,in=-30] (2,-2) to[out=-80,in=120] (3.5,-6) to[out=-60,in=160] (7,-8)
      to[out=160,in=-60] (2,-6)[in=-90,out=120] to (0,0) to[in=-90,out=-20] (10,0)
      to[out=90,in=60] (0,0);

      \draw[->] (0,0) to[in=90,out=60] (10,0); \draw (10,0)to[out=-90,in=-20] (0,0);
      \draw[->] (0,0) to[out=65,in=90] (13,0);
      \draw (13,0) to[out=-90,in=-30] (2,-2);
      
      \draw[-<] (2,-2) to[out=-80,in=120] (3.5,-6);
      \draw (3.5,-6) to[out=-60,in=160] (7,-8);

      \draw[-<] (0,0) to[out=-90,in=120] (2,-6);
      \draw (2,-6) to[out=-60,in=160] (7,-8);

    \draw[fill] (0,0) circle [radius=.3];
    \draw[fill, blue] (7,-8) circle [radius=.3];
    \draw[red] (2,-2) circle [radius=.3];
    \node[left] at (1.3,-5){$\gamma'$};
    \node[right] at (3.5,-5.3){$\eta'$};
    \node[below] at (6,3){$\gamma$};
    \node[above] at (10,3.3){$\eta$};
    
    \end{scope}

    \begin{scope}[shift={(77,5)}]
      \draw[->] (0,0) to[in=90,out=60] (10,0); \draw (10,0)to[out=-90,in=-20] (0,0);

      \draw[-<] (0,0) -- (2,-3);
      \draw (2,-3) --(6,-9);

    \draw[fill] (0,0) circle [radius=.3];
    \draw[fill,blue] (6,-9) circle [radius=.3];

    \end{scope}

  \end{tikzpicture}
  \caption{A saddle connection surgery along a closed saddle connection $\gamma$ and a twin
    $\eta$. The angle $\theta$ is responsible for the degree of the new full-dotted (blue) point.}
  \label{fig:s}
\end{figure}

\begin{rem}\label{remfa}
  If $X$ is in $\mathcal H_{2g-2}$, then a saddle connection surgery produces a translation surface
  with the same absolute period of $X$. If in addition the angle between the closed saddle
  connection 
  and the chosen twin is exactly $2\pi$, then the resulting 
  surface is in $\mathcal H_{2g-2}$ (the full-dotted blue point in Figure~\ref{fig:s} is a
  regular point). So, if different from
  $X$, it is a fake $X$. Moreover, the closed saddle connection used by the surgery, remains a
  closed saddle connection of the same length and direction after the surgery.
\end{rem}

\section{Iterated surgeries on the octagon}
In this section we describe a sequence of fake octagons $\oct_n$ obtained from the octagon
$\oct=\oct_0$ via a sequence of saddle connection surgeries. In particular, each surgery will
be a saddle connection surgery along a fixed closed saddle connection. We will then prove that
all fakes $\oct_n$ are in fact  different from each other.

We parameterise our octagon by gluing parallel sides of two polygons as in
Figure~\ref{figB1}. Edges have length one, all vertices are identified to each other and form the unique singular point.
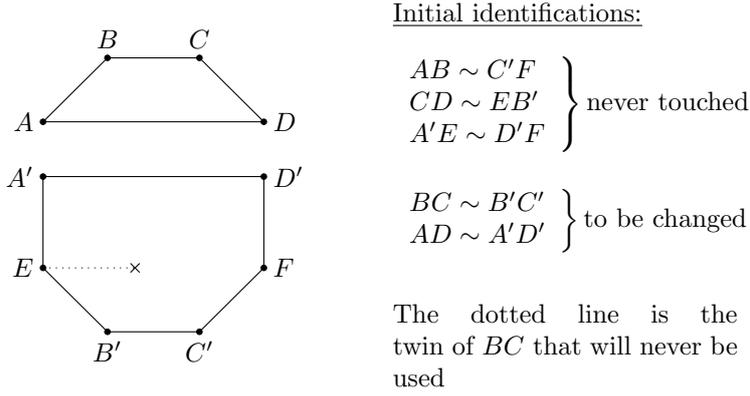
\begin{figure}[h]
\footnotesize  \centering
  \begin{tikzpicture}[x=.8ex,y=.8ex]
    \draw (17,17) -- (-7,17) -- (-7,7)--(0,0) -- (10,0) -- (17,7) -- (17,17);
    \draw (-7,23) -- (0,30)--(10,30)--(17,23)-- (-7,23);
    \foreach \a in{(0,0),(10,0),(17,7),(17,17), (-7,17), (-7,7),(-7,23), (0,30), (10,30),
      (17,23)} \draw[fill]  \a circle [radius  =.3];
    \draw[dotted] (-7,7)-- (3,7);\draw (3.5,6.5)--(2.5,7.5);\draw(3.5,7.5)--(2.5,6.5);
    \node[left] at (-7,17) {$A'$};
    \node[left] at (-7,23) {$A$};
    \node[left] at (-7,7) {$E$};
    \node[right] at (17,7) {$F$};
    \node[right] at (17,17) {$D'$};
    \node[right] at (17,23) {$D$};
    \node[below] at (0,0) {$B'$};
    \node[below] at (10,0) {$C'$};
    \node[above] at (0,30) {$B$};
    \node[above] at (10,30) {$C$};
    \node[right] at (30,15){\parbox{30ex}{ \underline{Initial identifications:}\\ \ \\
      $\left.\begin{array}{l} AB\sim C'F\\ CD\sim EB'\\ A'E\sim D'F\end{array}\right\}\text{never
        touched}$\\ \ \\ \ \\ $\left.\begin{array}{l} BC\sim B'C'\\
                                       AD\sim A'D' \end{array}\right\}\text{to be changed}$\\ \ \\
                                   \ \\
                                   The dotted line is the twin of $BC$ that will never be used
                                 }};
  \end{tikzpicture}
  \caption{The octagon}
  \label{figB1}
\end{figure}

The octagon has three horizontal (closed) saddle connections. Only one, which in the picture is
$BC$, has length $1$, and the other two $AD,EF$ have length $1+\sqrt2$. This property will be
preserved by all saddle connection surgeries. We therefore describe our surgeries from an
intrinsic viewpoint, exploiting this property.

Let $\gamma$ be the unique
unitary horizontal closed saddle connection, being the other two of length $1+\sqrt2$. By
definition of twin, the two twins of $\gamma$ are sub-segments of those longer saddle connections.
Since $\gamma$ is horizontal, the end of $\gamma$ forms with the 
start of $\gamma$ an angle which is an odd multiple of $\pi$. In fact for the octagon that
angle is $3\pi$. Since the total angle around the
singular point is $6\pi$, then the twins of
$\gamma$ form angles $\pm \pi$ with respect to the end of $\gamma$. We orient $\gamma$ from
left to right, and name $\gamma_L$ be the twin on the
``left side'', that is to say, the angle measured clockwise from the end of $\gamma$ to
$\gamma_L$ is $\pi$. Let $\gamma_R$ be the other twin. We define {\bf left surgery} the
saddle connection surgery along $\gamma$ and $\gamma_L$, and {\bf right surgery} that along
$\gamma$ and $\gamma_R$. (See also Figure~\ref{figBI}). The angle between $\gamma_L$ (or
$\gamma_R)$ and $\gamma$ is exactly $2\pi$, so left and right surgeries produce elements of
$\mathcal H_2$ (see Remark~\ref{remfa})
It is immediate to check that the inverse of a left surgery is a right surgery
along $\gamma^{-1}$. 

It will be clear from what follows that left and right surgeries preserve the two properties of 
having  one unitary  horizontal saddle connection (and two of length $1+\sqrt2$), and  that
the angle between the start and the end of $\gamma$ is $3\pi$. Therefore, we can iterate
left and right surgeries.

\begin{defn}
  For $n\in\Z$ we define $\oct_n$ as the translation surface obtained from the octagon $\oct_0$ by
  $n$ left surgeries (for negative $n$ we apply right surgeries).
\end{defn}

Before giving a global description of $\oct_n$, we start by looking in details at first steps.
Coming back to pictures, left surgeries will always affect the horizontal saddle connection
$\gamma=BC$ and its twin on the line $AD$. Specifically, the twin of $BC$ along $EF$ will never
come in play. Also, we never change diagonal identifications $AB\sim C'F, CD\sim EB'$, nor the vertical one
$A'E\sim D'F$.

Let's start.  We cut and paste along $BC$ and its twin on the line
$AD$. See Figure~\ref{figBI}.

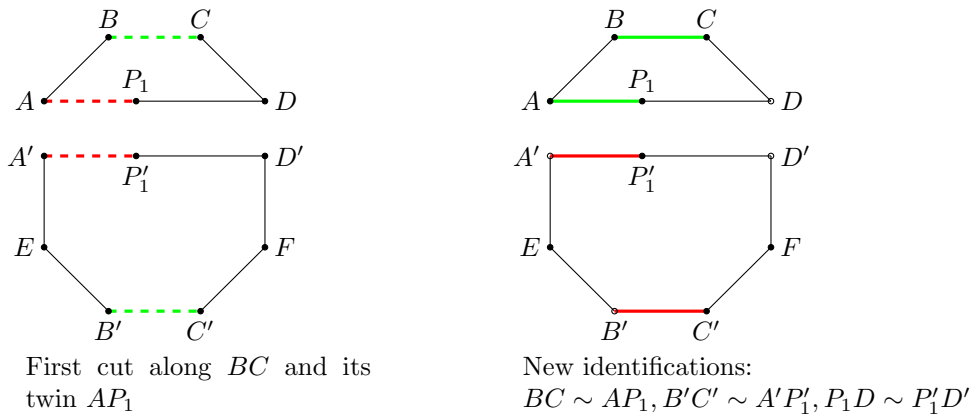
\begin{figure}[h]
\footnotesize  \centering
  \begin{tikzpicture}[x=.8ex,y=.8ex]
    \draw (-7,17)--(-7,7)--(0,0); \draw (10,0)--(17,7)-- (17,17);
    \draw (-7,23) -- (0,30); \draw (10,30)--(17,23);

        \draw[dashed, very thick, green] (0,0)--(10,0);
        \draw[dashed, very thick, green] (0,30)--(10,30);
        \draw[dashed, very thick, red] (-7,17)--(3,17);
        \draw[dashed, very thick, red] (-7,23)--(3,23);
        \draw (3,17)--(17,17);\draw (3,23)--(17,23);
        \foreach \a in{(3,17),(3,23)} \draw[fill]  \a circle [radius  =.3];

    \node[left] at (-7,17) {$A'$};
    \node[left] at (-7,23) {$A$};
         \node[above] at (3,23) {$P_1$};
         \node[below] at (3,17) {$P_1'$};
    \node[left] at (-7,7) {$E$};
    \node[right] at (17,7) {$F$};
    \node[right] at (17,17) {$D'$};
    \node[right] at (17,23) {$D$};
    \node[below] at (0,0) {$B'$};
    \node[below] at (10,0) {$C'$};
    \node[above] at (0,30) {$B$};
    \node[above] at (10,30) {$C$};

    \foreach \a in{(0,0),(10,0),(17,7),(17,17), (-7,17), (-7,7),(-7,23), (0,30), (10,30),
      (17,23)} \draw[fill]  \a circle [radius  =.3];

    \node[below right] at (-10,-4){\parbox{30ex}{\footnotesize First cut along $BC$ and its twin $AP_1$}};

    \begin{scope}[shift={(55,0)}]

     \draw (-7,17)--(-7,7)--(0,0); \draw (10,0)--(17,7)-- (17,17);
    \draw (-7,23) -- (0,30); \draw (10,30)--(17,23);

        \draw[very thick, red] (0,0)--(10,0);
        \draw[very thick, green] (0,30)--(10,30);
        \draw[very thick, red] (-7,17)--(3,17);
        \draw[very thick, green] (-7,23)--(3,23);
        \draw (3,17)--(17,17);\draw (3,23)--(17,23);
        \foreach \a in{(3,17),(3,23)} \draw[fill]  \a circle [radius  =.3];

    \node[left] at (-7,17) {$A'$};
    \node[left] at (-7,23) {$A$};
         \node[above] at (3,23) {$P_1$};
         \node[below] at (3,17) {$P_1'$};
    \node[left] at (-7,7) {$E$};
    \node[right] at (17,7) {$F$};
    \node[right] at (17,17) {$D'$};
    \node[right] at (17,23) {$D$};
    \node[below] at (0,0) {$B'$};
    \node[below] at (10,0) {$C'$};
    \node[above] at (0,30) {$B$};
    \node[above] at (10,30) {$C$};

    \foreach \a in{(10,0), (17,7), (-7,23), (-7,7), (0,30), (10,30)}
    \draw[fill]  \a circle [radius  =.3];

     \foreach \a in{(0,0),(-7,17), (17,23), (17,17)} \draw  \a circle [radius  =.3];
   
        \node[below right] at (-11,-4){\parbox{40ex}{\footnotesize New identifications: \\
            $BC\sim AP_1, B'C'\sim A'P_1', P_1D\sim P_1'D'$}};

 \end{scope}  
  \end{tikzpicture}
  \caption{First left surgery: first fake $\oct_1$.}
  \label{figBI}
\end{figure}
In that picture, dashed lines mean cuts, i.e. segments that where
previously identified and are no longer identified. Colours visualise new identifications.
Note that after the surgery, not all vertices are identified to each other. In particular, 
$A'\sim B'\sim D\sim D'$ is a regular point. All other vertices are identified, give rise to the unique
singular point, and the result is indeed a fake octagon: it is our $\oct_1$. We will label with a
full dot the singular point, and with other symbols those other vertices that are regular
points (we use same label for vertices that are identified).
Also, we will use the  ``dot'' notation for concatenation of segments, e.g. ``$XY\cdot
ZT$'' denotes the concatenation of segments $ZT$ after $XY$, clearly this makes sense only if
$Y$ is identified with $Z$.

When we cut the twin of $BC$ (oriented as $BC$) we see two avatars
of it in the picture: one with the surface on its left, and one on its right. We denote by
$P_1$ the endpoint of the cut having the surface on its left side, and $P_1'$ the other.

After the surgery, the saddle connection $BC$ has again two twins, one emanating from $P_1$ along the line $P_1D$ and another emanating from $E$.

We then obtain $\oct_2$ via a second left surgery, cutting and pasting along $BC$ and its twin
on the line $P_1D$. See Figure~\ref{figBII} (left side). As above, when cutting along that
twin, we denote by $P_2$ the endpoint of the cut having the surface in its left side, and
$P_2'$ the other. 
 
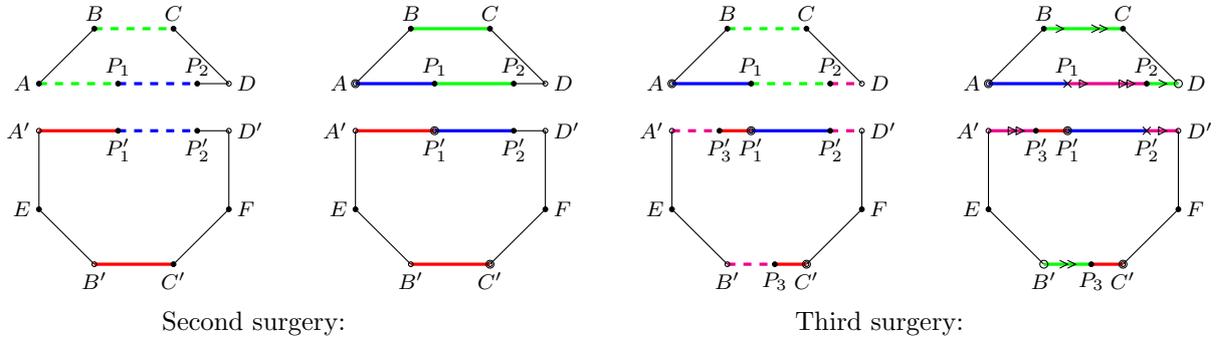
\begin{figure}[h]
\tiny  \centering
  \begin{tikzpicture}[x=.86ex,y=.86ex]
    \draw (-7,17)--(-7,7)--(0,0); \draw (10,0)--(17,7)-- (17,17);
    \draw (-7,23) -- (0,30); \draw (10,30)--(17,23);

        \draw[very thick, red] (0,0)--(10,0);
        \draw[dashed, very thick, green] (0,30)--(10,30);
        \draw[very thick, red] (-7,17)--(3,17);
        \draw[dashed, very thick, green] (-7,23)--(3,23);
        \draw[dashed, very thick, blue] (3,23)--(13,23);
        \draw[dashed, very thick, blue] (3,17)--(13,17);
        \draw (13,23)--(17,23);
        \draw (13,17)--(17,17);
        \foreach \a in{(3,17),(3,23),(13,17),(13,23)} \draw[fill]  \a circle [radius  =.3];

    \node[left] at (-7,17) {$A'$};
    \node[left] at (-7,23) {$A$};
         \node[above] at (3,23) {$P_1$};
         \node[below] at (3,17) {$P_1'$};
         \node[above] at (13,23) {$P_2$};
         \node[below] at (13,17) {$P_2'$};

    \node[left] at (-7,7) {$E$};
    \node[right] at (17,7) {$F$};
    \node[right] at (17,17) {$D'$};
    \node[right] at (17,23) {$D$};
    \node[below] at (0,0) {$B'$};
    \node[below] at (10,0) {$C'$};
    \node[above] at (0,30) {$B$};
    \node[above] at (10,30) {$C$};

    \foreach \a in{(10,0),(17,7), (-7,23), (-7,7), (0,30), (10,30)}
    \draw[fill]  \a circle [radius  =.3];

     \foreach \a in{(0,0),(-7,17), (17,23), (17,17)} \draw  \a circle [radius  =.3];

    \begin{scope}[shift={(40,0)}]
    \draw (-7,17)--(-7,7)--(0,0); \draw (10,0)--(17,7)-- (17,17);
    \draw (-7,23) -- (0,30); \draw (10,30)--(17,23);

        \draw[very thick, red] (0,0)--(10,0);
        \draw[very thick, green] (0,30)--(10,30);
        \draw[very thick, red] (-7,17)--(3,17);
        \draw[very thick, blue] (-7,23)--(3,23);
        \draw[very thick, green] (3,23)--(13,23);
        \draw[very thick, blue] (3,17)--(13,17);
        \draw (13,23)--(17,23);
        \draw (13,17)--(17,17);

    \node[left] at (-7,17) {$A'$};
    \node[left] at (-7,23) {$A$};
         \node[above] at (3,23) {$P_1$};
         \node[below] at (3,17) {$P_1'$};
         \node[above] at (13,23) {$P_2$};
         \node[below] at (13,17) {$P_2'$};

    \node[left] at (-7,7) {$E$};
    \node[right] at (17,7) {$F$};
    \node[right] at (17,17) {$D'$};
    \node[right] at (17,23) {$D$};
    \node[below] at (0,0) {$B'$};
    \node[below] at (10,0) {$C'$};
    \node[above] at (0,30) {$B$};
    \node[above] at (10,30) {$C$};

    \foreach \a in{(17,7), (-7,7),(13,17), (0,30),(10,30),(3,23),(13,23)} 
    \draw[fill]  \a circle [radius  =.3];

    \foreach \a in{(10,0),(3,17),(-7,23)} {\draw  \a circle [radius  =.3];\draw\a circle[radius=.5];};

        \foreach \a in{(-7,17), (17,23), (17,17),(0,0)} \draw  \a circle [radius  =.3];

 \end{scope}
\node[below] at (26,-5){\parbox{30ex}{\footnotesize Second surgery:}};

 \node[below right] at (-13,-12){\parbox{60ex}{\footnotesize On the left
     the cut along $BC$ and its twin $P_1P_2$.}};

      \node[below right] at (-13,-18){\parbox{60ex}{\footnotesize On the right the new identifications:\\ 
            ($B'C'\sim A'P_1', P_2D\sim P_2'D'$, and)\\ $BC\sim P_1P_2, AP_1\sim P_1'P_2'$.}};

 \begin{scope}[shift={(80,0)}]  
        \draw (-7,17)--(-7,7)--(0,0); \draw (10,0)--(17,7)-- (17,17);
    \draw (-7,23) -- (0,30); \draw (10,30)--(17,23);

        \draw[very thick, red] (6,0)--(10,0);
        \draw[dashed, very thick, green] (0,30)--(10,30);
        \draw[very thick, red] (-1,17)--(3,17);
        \draw[very thick, blue] (-7,23)--(3,23);
        \draw[dashed, very thick, green] (3,23)--(13,23);
        \draw[very thick, blue] (3,17)--(13,17);
        \draw[dashed, very thick, magenta] (13,23)--(17,23);
        \draw[dashed, very thick, magenta] (13,17)--(17,17);
        \draw[dashed, very thick, magenta] (-7,17)--(-1,17);
        \draw[dashed, very thick, magenta] (0,0)--(6,0);

    \node[left] at (-7,17) {$A'$};
    \node[left] at (-7,23) {$A$};
         \node[above] at (3,23) {$P_1$};
         \node[below] at (3,17) {$P_1'$};
         \node[above] at (13,23) {$P_2$};
         \node[below] at (13,17) {$P_2'$};
         \node[below] at (-1,17) {$P_3'$};
         \node[below] at (6,0) {$P_3$};

    \node[left] at (-7,7) {$E$};
    \node[right] at (17,7) {$F$};
    \node[right] at (17,17) {$D'$};
    \node[right] at (17,23) {$D$};
    \node[below] at (0,0) {$B'$};
    \node[below] at (10,0) {$C'$};
    \node[above] at (0,30) {$B$};
    \node[above] at (10,30) {$C$};

      \foreach \a in{(17,7), (-7,7),(13,17), (0,30),(10,30),(3,23),(13,23),(-1,17),(6,0)} 
    \draw[fill]  \a circle [radius  =.3];

    \foreach \a in{(10,0),(3,17),(-7,23)} {\draw  \a circle [radius  =.3];
      \draw\a circle[radius=.5];};

        \foreach \a in{(-7,17), (17,23), (17,17),(0,0)} \draw  \a circle [radius  =.3];

    \begin{scope}[shift={(40,0)}]

    \draw (-7,17)--(-7,7)--(0,0); \draw (10,0)--(17,7)-- (17,17);
    \draw (-7,23) -- (0,30); \draw (10,30)--(17,23);

        \draw[very thick, red] (6,0)--(10,0);
        \draw[very thick, green] (0,30)--(10,30);
        \draw[very thick, red] (-1,17)--(3,17);
        \draw[very thick, blue] (-7,23)--(3,23);
        \draw[very thick, magenta] (3,23)--(13,23);
        \draw[very thick, blue] (3,17)--(13,17);
        \draw[very thick, green] (13,23)--(17,23);
        \draw[very thick, magenta] (13,17)--(17,17);
        \draw[very thick, magenta] (-7,17)--(-1,17);
        \draw[very thick, green] (0,0)--(6,0);

        \foreach \a in {(0,30),(4.5,30),(5.5,30),(13,23),(0.5,0),(1.5,0)}
               {\begin{scope}[shift={\a}]
                   \draw (1.5,0.5)--(2.5,0)--(1.5,-0.5);
                 \end{scope}};

        \foreach \a in {(13,17),(-5,17),(-6,17),(3,23),(8,23),(9,23)}
               {\begin{scope}[shift={\a}]
                   \draw (1.5,0.5)--(2.5,0)--(1.5,-0.5)--(1.5,0.5);
                 \end{scope}};

    \node[left] at (-7,17) {$A'$};
    \node[left] at (-7,23) {$A$};
         \node[above] at (3,23) {$P_1$};
         \node[below] at (3,17) {$P_1'$};
         \node[above] at (13,23) {$P_2$};
         \node[below] at (13,17) {$P_2'$};
         \node[below] at (-1,17) {$P_3'$};
         \node[below] at (6,0) {$P_3$};

    \node[left] at (-7,7) {$E$};
    \node[right] at (17,7) {$F$};
    \node[right] at (17,17) {$D'$};
    \node[right] at (17,23) {$D$};
    \node[below] at (0,0) {$B'$};
    \node[below] at (10,0) {$C'$};
    \node[above] at (0,30) {$B$};
    \node[above] at (10,30) {$C$};

    \foreach \a in{(17,7), (-7,7), (0,30),(10,30),(13,23),(-1,17),(6,0)} 
    \draw[fill]  \a circle [radius  =.3];

    \foreach \a in{(10,0),(3,17),(-7,23)} {\draw  \a circle [radius  =.3];
      \draw\a circle[radius=.5];};

    \foreach \a in{(-7,17), (17,17)} \draw  \a circle [radius  =.3];
    \draw (12.5,17.5)--(13.5,16.5);\draw(12.5,16.5)--(13.5,17.5);
    \draw (2.5,23.5)--(3.5,22.5);\draw(2.5,22.5)--(3.5,23.5);
     \foreach \a in{(0,0),(17,23)} \draw  \a circle [radius  =.5];

 \end{scope}  
\node[below] at (26,-5){\parbox{30ex}{\footnotesize Third surgery:}};

 \node[below right] at (-17,-12){\parbox{65ex}{\footnotesize On the left
     the cut along $BC$ and its twin $P_2D\cdot B'P_3$.}};
 \node[below right] at (-17,-18){\parbox{65ex}{\footnotesize On the right the new identifications:\\ 
            ($AP_1\sim P_1'P_2', P_3'P_1'\sim P_3C'$, and)\\ $BC\sim P_2D\cdot B'P_3, P_1P_2\sim P_2'D'\cdot
            A'P_3'$.}};

 \end{scope}
\end{tikzpicture}
  \caption{Second and third fakes $\oct_2$ and $\oct_3$.}
  \label{figBII}
\end{figure}

One more left surgery, along $BC$ and its twin emanating from $P_2$, will produce $\oct_3$. See
Figure~\ref{figBII} (right side). Again, $P_3$ and $P_3'$ are the endpoints of the cut of the
twin having the surface on the left and right side respectively.

We are now ready to describe the gluing pattern of $\oct_n$. For this purpose it is more
convenient to pass to a 
simpler --- even if less ``octagonal'' --- viewpoint.  Namely, we glue the upper
quadrilateral to the bottom one, by identifying sides $AB$ and $C'F$. See
Figure~\ref{figBconfnonoct}.  (Compare also with~\cite[Figure 8]{Mcm14}). 
  \begin{figure}[h]
\footnotesize  \centering
    \begin{tikzpicture}[x=.96ex,y=.96ex]
    \draw (-7,17)--(-7,7)--(0,0); \draw (17,7)-- (17,17);\draw(27,7)--(34,0);

        \draw[very thick, blue] (0,0)--(34,0);
        \draw[dashed, very thick, green] (17,7)--(27,7);
                \draw[dashed, very thick, green] (6,0.2)--(16,0.2);
        \draw[very thick, blue] (-7,17)--(13,17);
        \draw[very thick, blue] (13,17)--(17,17);
        \draw[dotted] (10,0)--(17,7);
            
        \foreach\a in{(26,0)}
               {\begin{scope}[shift={\a}]
        \draw (-.5,-.5)--(.5,.5);\draw(-.5,.5)--(.5,-.5);
                 \end{scope}};

         \node[below] at (6,0) {$P_{n-1}$};
         \node[below] at (26,0) {$P_{n+1}$};
         \node[below] at (16,0) {$P_n$};
         \node[below] at (13,17) {$P_n'$};

    \node[left] at (-7,7) {$E$};
    \node[left] at (17,7) {$F\sim B$};
    \node[right] at (17,17) {$D'$};
    \node[right] at (34,0) {$D$};
    \node[below] at (0,0) {$B'$};
    \node[left] at (-7,17) {$A'$};    
    \node[above] at (27,7) {$C$};

    \foreach \a in{(17,7), (-7,7),(13,17),(27,7),(6,0),(16,0)} 
    \draw[fill]  \a circle [radius  =.5];
    
    \foreach \a in{(-7,17), (17,17)} \draw  \a circle [radius  =.3];
    \foreach \a in{(34,0),(0,0)} \draw  \a circle [radius  =.5];

      \end{tikzpicture}
    \caption{A less ``octagonal'' viewpoint. $P_n$ is
    identified with $P_n'$. $P_{n-1}P_n$ is where $BC$ is
    glued at step $n$, while $P_nP_{n+1}$ is the next twin we cut at step $n+1$. Segment
    $P_nD\cdot B'P_{n-1}$ is identified with $P_n'D'\cdot A'P_n'$. This is the $n^{th}$ fake $\oct_n$.}
    \label{figBconfnonoct}
  \end{figure}
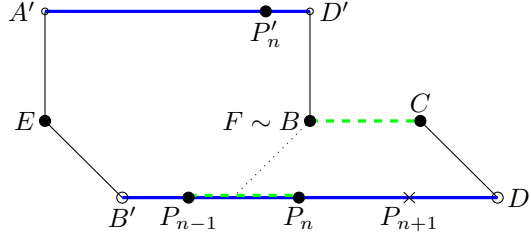

Horizontal gluings are determined, once we know positions of points
$P_n$ and $P_n'$, as follows. Since $B'$ is identified with $D$, segment $B'D$ can be
parameterised by a circle of length
$2+\sqrt2$. 
Points $P_{n-1}$ and $P_{n+1}$ are the points of the circle $B'D$ at distance $1$ from  $P_n$,
respectively on the left and right side of $P_n$.  
At step $n$, segment $BC$ is identified with $P_{n-1}P_n$ --- this is the unique unitary
horizontal saddle connection ---
and segment $P_n'D'\cdot A'P_n$ is
identified with $P_nP_{n-1}$ (which, in Figure~\ref{figBconfnonoct}, is the concatenation of
segments $P_nD\cdot B'P_{n-1}$), the latter being a horizontal saddle connection of length $1+\sqrt2$.
The third horizontal saddle connection, namely $EF$, is never involved and always has length
$1+\sqrt2$. The unique singular point is
$P_{n-1}\sim P_n\sim P_n'\sim B\sim C\sim E$, and a quick check shows that the angle between the start and the
end of the unitary closed horizontal saddle connection is $3\pi$.

The twin of $BC$ that will be used in next surgery is $P_nP_{n+1}$
(which is identified with the corresponding segment starting from $P_n'$), and  it is readily
checked that a left surgery along $BC$ and its twin $P_nP_{n+1}$ produces again a configuration
of the same type, with different positions of $P_n$ and $P_n'$:

  If we parameterise $B'D$ with $[0,2+\sqrt2]$ and $A'D'$ with $[0,1+\sqrt2]$, we see that
  $P_1=2$ and $P'_1=1$, and in general we have
  $$P_n\equiv n+1 \mod (2+\sqrt2)\qquad\qquad P'_n\equiv n \mod (1+\sqrt2).$$

  \begin{rem}\label{r1} Pictures only help in calculations, but
   left surgeries are intrinsically defined: any of our fakes has three horizontal saddle
   connections, and only one of them has length one. At any step we cut and paste along that saddle
   connection and its left twin. This receipt is ``picture free''.    
  \end{rem}
  \begin{thm}\label{thm1}
    If $n\neq m$, then $\oct_n\neq \oct_m$. 
  \end{thm}
  \begin{proof}
    The invariant that distinguishes fakes octagons from each other is the systole, namely
    the (family of) 
    shortest saddle connection(s). As the octagon has edge of length one, the systole is always
    not longer than one. In fact, the shortest saddle connections for the true octagon all have
    length one, and because the irrationality of $\sqrt2$ this never happens again.
    Looking at Figure~\ref{figBconfnonoct} we see that systoles are necessarily
    segments connecting some avatar of the singular point (i.e. $P_{n-1},P_n,P_n',E,B,C$). Point
    $P_n'$ always has distance at least one from other singular points, so no systole starts
    from $P_n'$ in Figure~\ref{figBconfnonoct}. Moreover, since the quadrilateral
    $P_{n-1}P_nCB$ is a parallelogram, for $n\neq 0$, we have three possible families of fakes
    octagons, 
    determined by the position of $P_n$ in $B'D=[0,2+\sqrt2]$ (see Figure~\ref{fig:sis}):
    \begin{enumerate}
    \item $P_n\in(1,1+\frac{1+\sqrt2}{2})$. The unique systole is the  segment $P_nB$.
    \item $P_n\in(1+\frac{1+\sqrt2}{2},2+\frac{1+\sqrt2}{2})$. There are two systoles: $P_{n-1}B$ and $P_nC$.

    \item $P_n\in(0,1)\cup(2+\frac{1+\sqrt2}{2},2+\sqrt2)$. In this case the unique systole is
      $P_{n-1}C$.  
    \end{enumerate}

  \begin{figure}[h]
  \centering
  \begin{tikzpicture}[x=1ex,y=1ex]
    \footnotesize
    \node at (-20,5) {$(1)$};
    \draw (-7,7)--(0,0); \draw(27,7)--(34,0);
    \draw(0,0)--(34,0);
    \draw (17,7)--(27,7);
              
    \draw[dotted] (10,0)--(17,7);
    \draw[dotted] (32,0)--(27,7)--(22,0)--(17,7)--(12,0);
    \foreach \a in {32,22,10,0} \draw (\a,0)--(\a,-1);
    \node[below] at (0,-1.5) {$0$};
    \node[below] at (10,-1.5) {$1$};
    \node[below] at (22,-1.5) {$1+\frac{1+\sqrt 2}{2}$};
    \node[below] at (32,-1.5) {$2+\frac{1+\sqrt 2}{2}$};        
           
         \node[below] at (6,0) {$P_{n-1}$};
         \node[below] at (16,0) {$P_n$};
         \draw[very thick, red, dashed] (16,0)--(17,7);
         
    \node[left] at (-7,7) {$E$};
    \node[above] at (17,7) {$B$};
    \node[right] at (34,0) {$D=2+\sqrt 2$};
    \node[left] at (-.5,0) {$B'$};
    \node[above] at (27,7) {$C$};

    \foreach \a in{(17,7), (-7,7),(27,7),(6,0),(16,0)} 
    \draw[fill]  \a circle [radius  =.3];

    \begin{scope}[shift={(0,-15)}]
      \node at (-20,5) {$(2)$};
      \draw (-7,7)--(0,0); \draw(27,7)--(34,0);
    \draw(0,0)--(34,0);
    \draw (17,7)--(27,7);
              
    \draw[dotted] (10,0)--(17,7);
    \draw[dotted] (32,0)--(27,7)--(22,0)--(17,7)--(12,0);
    \foreach \a in {32,22,10,0} \draw (\a,0)--(\a,-1);
    \node[below] at (0,-1.5) {$0$};
    \node[below] at (10,-1.5) {$1$};
    \node[below] at (22,-2) {$1+\frac{1+\sqrt 2}{2}$};
    \node[below] at (32,-1.5) {$2+\frac{1+\sqrt 2}{2}$};        
           
         \node[below] at (15,0) {$P_{n-1}$};
         \node[below] at (25,0) {$P_n$};
         \draw[very thick, red, dashed] (15,0)--(17,7);
         \draw[very thick, red, dashed] (25,0)--(27,7);
         
    \node[left] at (-7,7) {$E$};
    \node[above] at (17,7) {$B$};
    \node[right] at (34,0) {$D=2+\sqrt 2$};
    \node[left] at (-.5,0) {$B'$};
    \node[above] at (27,7) {$C$};

    \foreach \a in{(17,7), (-7,7),(27,7),(15,0),(25,0)} 
    \draw[fill]  \a circle [radius  =.3];

    \end{scope}

    \begin{scope}[shift={(0,-32)}]
      \node at (-20,5) {$(3)$};
          \draw (-7,7)--(0,0); \draw(27,7)--(34,0);
    \draw(0,0)--(34,0);
    \draw (17,7)--(27,7);
              
    \draw[dotted] (10,0)--(17,7);
    \draw[dotted] (32,0)--(27,7)--(22,0)--(17,7)--(12,0);
    \foreach \a in {32,22,10,0} \draw (\a,0)--(\a,-1);
    \node[below] at (0,-1.5) {$0$};
    \node[below] at (10,-1.5) {$1$};
    \node[below] at (22,-1.5) {$1+\frac{1+\sqrt 2}{2}$};
    \node[below] at (32,-1.5) {$2+\frac{1+\sqrt 2}{2}$};        
           
         \node[below] at (28,0) {$P_{n-1}$};
         \node[below] at (4,0) {$P_n$};
         \draw[very thick, red, dashed] (28,0)--(27,7);
    
    \node[left] at (-7,7) {$E$};
    \node[above] at (17,7) {$B$};
    \node[right] at (34,0) {$D=2+\sqrt 2$};
    \node[left] at (-.5,0) {$B'$};
    \node[above] at (27,7) {$C$};

    \foreach \a in{(17,7), (-7,7),(27,7),(4,0),(28,0)} 
    \draw[fill]  \a circle [radius  =.3];

    \end{scope}
      \end{tikzpicture}
    \caption{The three possible systole configurations.}
    \label{fig:sis}
  \end{figure}
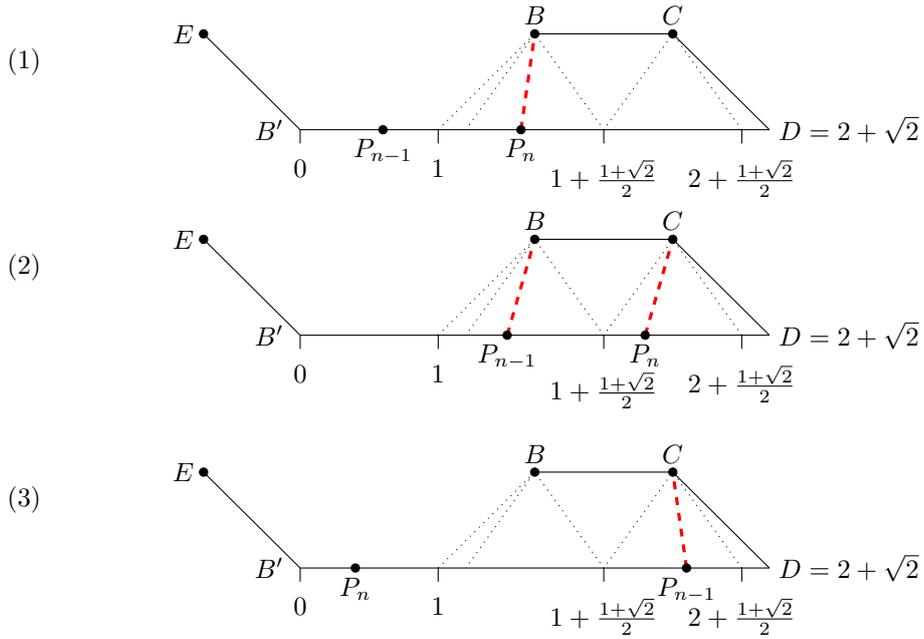

Since $2+\sqrt2$ is irrational and $P_n\equiv n+1 \mod(2+\sqrt2)$, the possible
positions of $P_n$ on $B'D$ identified with $[0,2+\sqrt2]$, form an infinite dense set. It
follows that the set of lengths of systoles of the family $\{\oct_n;n\in\Z\}$ is an infinite
set. Hence, the family of fakes $\{\oct_n:n\in\Z\}$ contains infinitely many different fakes.

Suppose now that there is $n,m$ such that $\oct_n=\oct_m$. Then (by Remark~\ref{r1}) in this
case, also $\oct_{n+i}=\oct_{m+i}$ for any $i$, and so we would observe a $m-n$ periodic
behaviour. In particular we would have only finitely many fakes among our $\oct_n$'s. But,
since we already proved that we have infinitely many different fakes, this cannot
happen. It follows that for any $n\neq m$ we have $\oct_n\neq \oct_m$.
\end{proof}

  \begin{rem}\label{r44}
     The fact that the possible positions of $P_n$ in $[0,2\sqrt2]$ form an infinite
 dense set, implies in particular that all possibilities described in Theorem~\ref{thm1}
 actually arise. Another consequence is that we can find fakes $\oct_n$ arbitrarily close to
 the octagon $\oct_0$, and in general that for any $\oct_m$ there is a fake $\oct_n$
 arbitrarily close to, but different from, $\oct_m$. This is nothing but a manifestation of 
 general density phenomena described in~\cite{CDF} and anticipated in Introduction.  
  \end{rem}

  \begin{rem}
  Even if any $\oct_n$ is different from each other, the systoles may have the same
  length. For instance, if $1+\frac{\sqrt2 -1}{2}<P_n<1+\frac{\sqrt2 +1}{2} 
  \mod (2+\sqrt2)$, then
  $\oct_n,\oct_{n+1},$ and $\oct_{n+2}$ have the systole(s) of the same length (the three being in
  families $(1),(2),(3)$ respectively).  
  \end{rem}

  This is basically all that can happens.
  
  \begin{prop} For any $\oct_m$ (with $m\neq 0$) there is $\oct_n$
    with the same systole length and in family $(1)$, more precisely with $P_n\equiv
    x\in(1,1+\frac{\sqrt2}{2})\mod(2+\sqrt2)$. Moreover, 
    \begin{itemize}
    \item if $P_n\in(\frac{1+\sqrt2}{2},1+\frac{\sqrt2}{2})  \mod(2+\sqrt2)$, then
      $\oct_m$ has the same systole-length of $\oct_n$ if and only if $m=\pm n, \pm n+1,\pm n+2$;
    \item if $P_n\in(1,\frac{1+\sqrt2}{2}) \mod(2+\sqrt2)$, then $\oct_m$ has the
      same systole-length of $\oct_n$ if and only if  $m=n$ or $m=-n+2$. 
    \end{itemize}

  \end{prop}
  \begin{proof}
      \begin{figure}[h]
  \centering
  \begin{tikzpicture}[x=1.5ex,y=1.5ex]
    \footnotesize
    \draw (-7,7)--(0,0); \draw(27,7)--(34,0);
    \draw(0,0)--(34,0);
    \draw (17,7)--(27,7);
              
    \draw[dotted] (10,0)--(17,7)--(17,0);
    \draw[dashed, very thick, red] (14,0)--(17,7)--(20,0);
    \draw[dashed, very thick, red] (24,0)--(27,7)--(30,0);
    \draw[dotted] (32,0)--(27,7)--(22,0)--(17,7)--(12,0);
    \draw[dotted] (27,7)--(27,0);
    \foreach \a in {34,32,22,10,0,12,17,27} \draw (\a,0)--(\a,-1);
    \node[below] at (0,-1.5) {\tiny $0$};
    \node[below] at (10,-1.5) {\tiny $1$};
    \node[below] at (22,-1.5) {\tiny $1+\frac{1+\sqrt 2}{2}$};
    \node[below] at (32,-1.5) {\tiny $2+\frac{1+\sqrt 2}{2}$};
    \node[below] at (12,-1.5) {\tiny $\frac{1+\sqrt 2}{2}$};
    \node[below] at (17,-1.5) {\tiny $1+\frac{\sqrt 2}{2}$};
    \node[below] at (27,-1.5) {\tiny $2+\frac{\sqrt 2}{2}$};        
    \node[below] at (14,0) {\tiny $x$};
    \node[below] at (20,0) {\tiny $y$};
    \node[below] at (24,0) {\tiny $z$};
    \node[below] at (30,0) {\tiny $t$};       

    \node[left] at (-7,7) {$E$};
    \node[above] at (17,7) {$B$};
    \node[right] at (34,0) {\tiny $D=2+\sqrt 2$};
    \node[left] at (-.5,0) {$B'$};
    \node[above] at (27,7) {$C$};

    \foreach \a in{(17,7), (-7,7),(27,7)}
    \draw[fill]  \a circle [radius  =.3];

      \end{tikzpicture}
    \caption{Positions having the same distance form $B$ or $C$}
    \label{fig:sl}
  \end{figure}
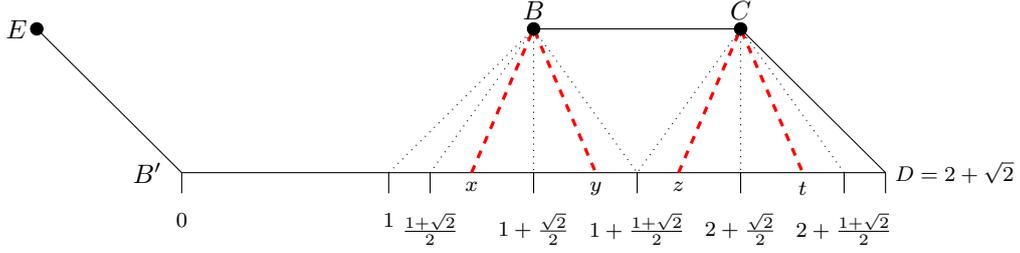
For $x\in [0,2+\sqrt2]$ let $y=y(x)$ be its symmetric with respect to $1+\sqrt2/2$. This is the
unique other point so that $d(x,B)=d(y,B)$. Explicitly, $y$ is determined
by $$\frac{x+y}{2}=1+\frac{\sqrt2}{2}\qquad \text{whence} \qquad x+y=2+\sqrt2.$$
Let $z=z(x)=x+1$ and $t=t(x)=y(x)+1$. Those are the unique
points so that $d(x,B)=d(z,C)=d(t,C)$.
Note that
\begin{equation}
  \label{eq:1}
x\equiv -y\equiv z-1\equiv -t+1 \mod(2+\sqrt2).  
\end{equation}

Such equations have integer coefficient and $2+\sqrt2$ is irrational. So, if we want to solve
them in $\Z$, they reduce to genuine equalities. Namely, if $x\equiv P_n\equiv
n+1\mod(2+\sqrt2)$ and $y\equiv P_m\equiv m+1\mod(2+\sqrt2)$, then $x\equiv -y\mod(2+\sqrt2)$ if
and only if $m=-n$, and similarly for points $z,t$.

The first consequence of this fact is that if $P_m$ is placed in
$(1+\frac{\sqrt2}{2},2+\sqrt2)$, then there is $n$ such that $P_n$ is placed in
$x\in(1,1+\frac{\sqrt2}{2})$ (hence $\oct_n$ is in family $(1)$) and  $P_m$ is either a $y$- or
$z$- or $t$-point for $x$. In particular, this proves the first claim.

We may therefore assume that we have $\oct_n$ in family $(1)$ and search for all  possible
$\oct_m$ with the same systole-length.

From the fact that congruences~\ref{eq:1} reduces to genuine identities on $\Z$, we can now deduce
second claims.

If $P_n\equiv x\in(\frac{1+\sqrt2}{2},1+\frac{\sqrt2}{2})\mod(2+\sqrt2)$, then the possibility
for $\oct_m$ to have the same systole-length as $\oct_n$ are two for each family, and precisely:

\begin{itemize}
  \item $\oct_m$ is in family $(1)$:
  \begin{itemize}
  \item $P_m$ coincides with $x$. This is possible only if $m=n$ 
\item $P_m\equiv y= -x \mod (2+\sqrt2)$, which happens if and only if $m=-n$;
\end{itemize}
\item $\oct_m$ is in family $(2)$:
  \begin{itemize}
\item $P_m\equiv z\equiv x+1 \mod (2+\sqrt2)$, which happens if and only if $m=n+1$. In this case $P_{m-1}\equiv x\mod (2+\sqrt2)$;
\item $P_m\equiv t\equiv -x+1 \mod (2+\sqrt2)$, which happens if and only if $m=-n+1$. In this case  $P_{m-1}\equiv y\equiv -x\mod (2+\sqrt2)$;
\end{itemize}
\item $\oct_m$ is in family $(3)$:
  \begin{itemize}
\item $P_{m-1}\equiv z\equiv x+1\mod (2+\sqrt2)$, which happens if and only if $m=n+2$;
\item $P_{m-1}\equiv t\equiv -x+1\mod (2+\sqrt2)$, which happens if and only if $m=-n+2$.
\end{itemize}
\end{itemize}
If $P_n\equiv x\in(1,\frac{1+\sqrt2}{2})\mod(2+\sqrt2)$, some possibility disappears because
in this case $d(y,B)>d(y,C)$ and $d(z,C)>d(z,B)$. A part the case $\oct_m=\oct_n$ (if and only
if $m=n$), the only possibility that remains is when $\oct_m$ belongs to family $(3)$ and $P_m$
is the $t$-point of $x\equiv P_n\mod (2+\sqrt2)$, namely:

\begin{itemize}
\item $P_{m-1}\equiv t\equiv -x+1\mod (2+\sqrt2)$, and this happens if and only if $m=n-n+2$.
\end{itemize}
 \end{proof}

 \begin{rem}[Generalisations]
   The construction of sequence $(\oct_n)_{n\in \Z}$ used only the existence of 
   a (horizontal) saddle connection $\gamma$ having an embedded twin such that:
   \begin{itemize}
   \item The angle from the start of the twin to the start of $\gamma$ is $2\pi$. (So that the
     saddle connection surgery produces a point in the minimal stratum, see Remark~\ref{remfa}.)
   \item The angle from the end of $\gamma$ to the start of the twin is $\pi$. 
   \item If the (horizontal) continuation of the twin is a saddle connection (which is
 longer than $\gamma$ because the twin is embedded), then  the angle from its start to its end
 is $\pi$ (hence it bounds a cylinder). 
\end{itemize}
Second condition implies that first one is preserved by the surgery; third condition is
preserved by surgery and guarantees that the length of the twin saddle connection does not
change under the surgery (to see this, just draw the twin and angles in Figure~\ref{fig:s}).

  Therefore the sequence of (putative) fakes can be constructed in any such situation via left
  surgeries. 
 \end{rem}

\end{document}